\documentclass[11pt]{article}

\usepackage[a4paper,margin=1.1in]{geometry}
\usepackage{setspace}
\usepackage[T1]{fontenc}
\usepackage{lmodern}
\usepackage{microtype}

\usepackage{amsmath,amssymb,amsthm}
\usepackage{mathtools}
\usepackage{mathrsfs}
 \usepackage{pgf,tikz}
 \usetikzlibrary{automata}
 \usetikzlibrary{arrows}
\usepackage[numbers]{natbib}

\usepackage[hidelinks]{hyperref}

\usepackage{fancyhdr}
\theoremstyle{plain}
\newtheorem{theorem}{Theorem}[section]
\newtheorem{lemma}[theorem]{Lemma}
\newtheorem{proposition}[theorem]{Proposition}

\theoremstyle{definition}
\newtheorem{definition}[theorem]{Definition}

\theoremstyle{remark}
\newtheorem{remark}[theorem]{Remark}

\newcommand{\R}{\mathbb{R}}
\newcommand{\PP}{\mathbb{P}}
\newcommand{\C}{\mathbb{C}}

\newcommand{\mdr}{{\rm mdr}}
\newcommand{\LL}{\mathcal{L}}

\title{\textbf{On the boundedness of some real line arrangements of type at most one}}
\author{Marek Janasz}
\date{\today}

\begin{document}
\maketitle

\begin{abstract}
In this note, we show that real line arrangements of type at most one, admitting only intersection points of multiplicity at most five, satisfy certain boundedness properties. In particular, we prove that a free real arrangement of \(d\) lines with intersection multiplicities bounded by \(5\) can have at most \(522\) lines and consequently there exist only finitely many combinatorial types of such arrangements.
\end{abstract}
\section{Introduction}\label{s:introduction}
The study of line arrangements in the real and complex projective plane lies at the intersection
of combinatorics, algebraic geometry, and singularity theory. Over the past decades, considerable
attention has been devoted to understanding how algebraic properties of arrangements, such as
freeness or plus-one generated, interact with their combinatorial structure, in particular with the
distribution and multiplicities of intersection points.

A central notion in this area is that of a \emph{free plane curve}, introduced by K.~Saito \cite{Saito} in a general setting and
later developed extensively for hyperplane arrangements. The freeness can be characterized algebraically via the module of Jacobian syzygies or, equivalently,
via the splitting type of the sheaf of logarithmic vector fields \cite{Dimca}. Free arrangements enjoy
remarkable structural properties, yet their existence is strongly constrained by combinatorial and geometric conditions. One of the fundamental topics in the subject is to determine whether certain classes of arrangements are \emph{bounded}, that is, whether there exists an absolute
upper bound on the number of lines once suitable restrictions on the intersection multiplicities
are imposed.

For real line arrangements, the situation is particularly subtle due to the presence of
additional geometric inequalities, such as Melchior’s or Shnurnikov's inequality \cite{Melchior, Sh}, which have no direct complex
counterpart. These inequalities often lead to strong restrictions on the possible combinatorics
of real arrangements and have been successfully used to prove boundedness results in several
contexts. In particular, boundedness results are known for real free line arrangements with
double, triple and quadruple points \cite{JanLes05}.

More recently, the class of \emph{plus-one generated} curves has emerged as a natural
generalization of free and nearly free curves. Introduced and systematically studied by Abe in the setting of hyperplane arrangements \cite{Abe18}, and later by Dimca and Sticlaru \cite{DS20} in the context of plane curves, these objects form the next level in the hierarchy of plane curves defined by the complexity of their Jacobian syzygies. From the perspective of line arrangements, this class provides a natural testing ground for investigating how far boundedness phenomena extend beyond the free case.

The purpose of the present paper is to establish new boundedness results for real line
arrangements of type at most one under explicit constraints on the multiplicities of intersection
points. Our first main result concerns real free line arrangements whose intersection points have
multiplicity at most five. We prove that such arrangements are bounded, and more precisely, that
they consist of at most $522$ lines. The proof combines algebraic constraints coming from
freeness, refined combinatorial counting arguments, and classical inequalities for real line
arrangements.

Our second main result addresses real plus-one generated line arrangements admitting only double,
triple and quadruple intersection points. In this setting, we prove a substantially stronger
bound, showing that the number of lines is at most $47$. This demonstrates that boundedness
persists beyond the class of free arrangements and highlights the rigidity imposed by the
plus-one generated condition when combined with low intersection multiplicities.

The paper is organized as follows. In Section~2 we recall the necessary background on Jacobian
syzygies, Milnor algebras, and the notions of freeness, plus-one generation, and the type of a
plane curve. Section~3 is devoted to real free line arrangements with points of multiplicity at
most five, where we first establish combinatorial bounds on high-multiplicity intersection
points and then prove the main boundedness theorem. In Section~4 we treat the case of real
plus-one generated line arrangements with intersection multiplicities at most four and obtain
the corresponding boundedness result. 
\section{Preliminaries}

We follow the notation introduced in~\cite{Dimca}.
Let us denote by $S := \mathbb{C}[x,y,z]$ the homogeneous coordinate ring of
$\mathbb{P}^2_{\mathbb{C}}$.
Let $C \subset \mathbb{P}^2_{\mathbb{C}}$ be a reduced plane curve of degree $d$
defined by a homogeneous polynomial $f \in S_d$.
We denote by $J_f$ the Jacobian ideal of $f$ which is the ideal generated by the all partials, i.e.,
\[
J_f = \langle f_x, f_y, f_z \rangle.
\]
We consider the graded $S$-module of Jacobian syzygies of $f$, defined by
\[
AR(f) = \{ (a,b,c) \in S^3 : a f_x + b f_y + c f_z = 0 \}.
\]
The minimal degree of a non-trivial Jacobian relation of $f$ is defined as
\[
\mdr(f) := \min \{ r \ge 0 : AR(f)_r \neq 0 \}.
\]
The Milnor algebra associated with $f$ is defined as
\[
M(f) := S / J_f.
\]

\begin{definition}
A reduced plane curve $C$ is called an \emph{$m$-syzygy curve} if its Milnor
algebra $M(f)$ admits a minimal graded free resolution of the form
\[
0 \longrightarrow \bigoplus_{i=1}^{m-2} S(-e_i)
\longrightarrow \bigoplus_{i=1}^{m} S(1-d-d_i)
\longrightarrow S^3(1-d)
\longrightarrow S
\longrightarrow M(f)
\longrightarrow 0,
\]
where the integers satisfy
$
e_1 \le e_2 \le \ldots \le e_{m-2}
\quad \text{and} \quad
1 \le d_1 \le \ldots \le d_m.
$
The ordered $m$-tuple $(d_1,\ldots,d_m)$ is called the set of \emph{exponents} of
$C$.
\end{definition}

\begin{definition}[{\cite[Definition 1.1]{her}}]
Let $C =\{f = 0\}\subset \mathbb{P}^{2}_{\mathbb{C}}$ be a reduced curve of degree $d$, and let $d_1$ and $d_2$ be the two smallest degrees of a minimal system of generators of the module of Jacobian syzygies $AR(f)$.  
We say that $C$ has \emph{type} $t(C)$, where the positive integer $t(C)$ is defined as
\[t(C) = d_1 + d_2 + 1 - d.\]
\end{definition}
The following definitions come from \cite[Theorem 1.1]{her}.
\begin{definition}
A reduced plane curve $C \subset \mathbb{P}^2_{\mathbb{C}}$ is said to be \emph{free} with the exponents $(d_{1},d_{2})$ if and only if $t(C)=0$.
\end{definition}
\begin{definition}
A reduced plane curve $C \subset \mathbb{P}^2_{\mathbb{C}}$ is said to be \emph{plus-one generated} with the exponents $(d_1,d_2,d_3)$ if and only if $t(C)=1$.
\end{definition}
In order to decide whether a given curve either is free or plus-one generated, we can use the following crucial results.

\begin{theorem}[{du Plessis -- Wall, \cite{duP}}]
\label{dup}
Let $C = \{f=0\}\subset \mathbb{P}^{2}_{\mathbb{C}}$ be a reduced curve of degree $d$ and let $d_1 = {\rm mdr}(f)$. Let us denote the total Tjurina number of $C$ by $\tau(C)$.
Then the following two cases hold.
\begin{enumerate}
\item[a)] If $d_1 < d/2$, then $\tau(C) \leq \tau_{max}(d,d_1 )= (d-1)(d-d_1-1)+d_{1}^2$ and the equality holds if and only if the curve $C$ is free.
\item[b)] If $d/2 \leq d_{1} \leq d-1$, then
$\tau(C) \leq \tau_{max}(d,d_{1})$,
where, in this case, we set
$$\tau_{max}(d,d_{1})=(d-1)(d-d_{1}-1)+d_{1}^2- \binom{2d_{1}-d+2}{2}.$$
\end{enumerate}
\end{theorem}

\begin{proposition}[{Dimca -- Sticlaru, \cite{DS20}}]\label{prop:DimStr}
Let $C \subset \mathbb{P}^2_\C$ be a reduced plane curve defined by $f=0$ of degree $d \ge 3$, admitting exactly three minimal syzygies with exponents $(d_1,d_2,d_3)$.  
Then $C$ is plus-one generated if and only if its total Tjurina number satisfies
\[
\tau(C) = (d-1)^2 - d_1(d-d_1-1) - (d_3-d_2+1).
\]
\end{proposition}

\section{The boundedness of real free line arrangements with points of multiplicity at most five}
 We will use the following inequality that comes from \cite{Sh}. For a line arrangement $\mathcal{L}$ we denote by $t_{k} = t_{k}(\mathcal{L})$ the number of $k$-fold intersection points in $\mathcal{L}$.
\begin{theorem}[Shnurnikov]
Let $\mathcal{L} \subset \mathbb{P}^{2}_{\mathbb{R}}$ be an arrangement of $d$ lines such that $t_{d}=t_{d-1}=t_{d-2} = 0$. Then the following inequality holds:
\[ t_2 + \frac{3}{2}t_3 \;\ge\; 8 + \sum_{r\ge 4}\frac{4r-15}{2}\,t_r.\]
\end{theorem}
We start with a general result of independent interest.
\begin{lemma}\label{thm:tk_global_bound}
Let $\mathcal \LL$ be an arrangement of $d$ lines in $\mathbb{P}^2_\R$.
Fix an integer $k\ge 4$ and assume that all intersection points have multiplicity at most $k$,
i.e.\ $t_i(\mathcal L)=0$ for all $i\ge k+1$. Then the number of $k$--fold
intersection points $t_{k}$ satisfies the following bounds.

\smallskip
\noindent\textbf{(1) Case $d\ge  k+3$.} One has
\[
t_k \;\le\; \frac{d(d-1)-16}{k^2+3k-15}.
\]
\smallskip
\noindent\textbf{(2) General case.} One has
\[
t_k \;\le\; \frac{d}{k}\left\lfloor\frac{d-1}{k-1}\right\rfloor .
\]
\smallskip
In particular, for every $d\ge 2$ we have
\[
t_k \;\le\;
\max\!\left\{
\frac{d}{k}\left\lfloor\frac{d-1}{k-1}\right\rfloor,\;
\left\lfloor\frac{d(d-1)-16}{k^2+3k-15}\right\rfloor
\right\},
\]
where the second term makes sense only when $d\ge k+3$.
\end{lemma}

\begin{proof}
Counting unordered pairs of distinct lines in two ways yields
\begin{equation}\label{eq:pairs_general_k}
\binom{d}{2} \;=\; \sum_{r\ge 2}\binom{r}{2}\,t_r.
\end{equation}
Under the assumption $t_r=0$ for $r\ge k+1$, identity \eqref{eq:pairs_general_k} becomes
\begin{equation}\label{eq:pairs_trunc_general_k}
\binom{d}{2} \;=\; t_2 + \sum_{r=3}^{k}\binom{r}{2}\,t_r,
\end{equation}
hence
\begin{equation}\label{eq:t2_general_k}
t_2 \;=\; \binom{d}{2} - \sum_{r=3}^{k}\binom{r}{2}\,t_r.
\end{equation}

\medskip
\paragraph{The case $d\ge k+3$.}
Then the maximal multiplicity satisfies $m(\mathcal L)\le k\le d-3$, so we may apply Shnurnikov's inequality
\begin{equation}\label{eq:HL_general_k}
t_2 + \frac{3}{2}t_3 \;\ge\; 8 + \sum_{r\ge 4}\frac{4r-15}{2}\,t_r.
\end{equation}
Since $t_r=0$ for $r\ge k+1$, \eqref{eq:HL_general_k} reduces to
\begin{equation}\label{eq:HL_trunc_general_k}
t_2 + \frac{3}{2}t_3 \;\ge\; 8 + \sum_{r=4}^{k}\frac{4r-15}{2}\,t_r.
\end{equation}
Substituting \eqref{eq:t2_general_k} into \eqref{eq:HL_trunc_general_k} gives
\[
\binom{d}{2}
-
\sum_{r=3}^{k}\binom{r}{2}\,t_r
+
\frac{3}{2}t_3
\;\ge\;
8 + \sum_{r=4}^{k}\frac{4r-15}{2}\,t_r,
\]
hence, after rearranging,
\begin{equation}\label{eq:rearranged_general_k}
\binom{d}{2} - 8
\;\ge\;
\frac{3}{2}t_3
+
\sum_{r=4}^{k}
\left(
\binom{r}{2} + \frac{4r-15}{2}
\right)t_r.
\end{equation}
All terms on the right-hand side are nonnegative. In particular,
\[
\binom{d}{2} - 8
\;\ge\;
\left(
\binom{k}{2} + \frac{4k-15}{2}
\right)t_k.
\]
A direct computation yields
\[
\binom{k}{2} + \frac{4k-15}{2}
=
\frac{k(k-1)}{2} + \frac{4k-15}{2}
=
\frac{k^2+3k-15}{2}.
\]
Therefore,
\[
t_k
\;\le\;
\frac{2\bigl(\binom{d}{2}-8\bigr)}{k^2+3k-15}
=
\frac{d(d-1)-16}{k^2+3k-15},
\]
which proves (1).

\medskip
\paragraph{General bound.}
Fix a line arrangement $\LL$. For every line $\ell$ in $\LL$, the following obvious count holds
\[(d-1) = \sum_{p \in {\rm Sing}(\mathcal{L}) \cap \ell} (r-1)t_{r,\ell},\]
where $t_{r,\ell}$ denotes the number of $r$-fold points located on the line $\ell$.
The above formulate tells us that each $k$--fold point from ${\rm Sing}(\mathcal{L}) \cap \ell$ costs $k-1$ intersections (since $k-1$ other lines meet $\LL$ at that point). Hence the number of $k$--fold points in ${\rm Sing}(\mathcal{L}) \cap \ell$ is at most $\lfloor (d-1)/(k-1)\rfloor$.
Counting incidences between lines and $k$--fold points gives us
\[
k\,t_k \;\le\; d\left\lfloor\frac{d-1}{k-1}\right\rfloor,
\]
and thus
\[
t_k \;\le\; \frac{d}{k}\left\lfloor\frac{d-1}{k-1}\right\rfloor,
\]
which proves (2). The final unified estimate follows immediately.
\end{proof}

\begin{remark}
For $k=5$, part (1) specializes to
\[
t_5 \le \frac{d(d-1)-16}{25}.
\]
\end{remark}
For the proof of the next result, which is our main contribution in the present paper, we need to recall Melchior's inequality \cite{Melchior}.
\begin{theorem}[Melchior]
Let $\mathcal{L} \subset \mathbb{P}^{2}_{\mathbb{R}}$ be an arrangement of $d\geq 3$ lines such that $t_{d}=0$. Then
\begin{equation}
\label{Melchior}
t_{2} \geq 3 + t_{4} + 2t_{5} + 3t_{6} + \ldots + (d-4)t_{d-1}.  
\end{equation}
\end{theorem}
\begin{theorem}
\label{main}
Let $\mathcal{L} \subset \mathbb{P}^{2}_{\mathbb{R}}$ be an arrangement of $d\geq 8$ lines such that $t_{d}=0$. Assume that $\mathcal{L}$ is free. Then $d\leq 522$.
\end{theorem}
\begin{proof}
We assumed that $\mathcal{L}$ is free which implies that the following equality holds:
\[ d_{1}^{2} - d_{1}(d-1) +(d-1)^2 = \tau(\mathcal{L}) = t_{2} + 4t_{3} + 9t_{4}+16t_{5},\]
where $d_{1}={\rm mdr}(\mathcal{L})$. After applying the naive combinatorial count we obtain
\[d_{1}^{2}-d_{1}(d-1) + (t_{2}+2t_{3}+3t_{4}+4t_{5}-d+1) = 0.\]
We compute the discriminant of the above equation and we obtain
\begin{equation}
\label{cruc}
\triangle_{d_{1}} := (d-1)^2 - 4(1-d) -4(t_{2}+2t_{3}+3t_{4}+4t_{5}).
\end{equation}
Since $\mathcal{L}$ is free we have $\triangle_{d_{1}}\geq 0$, and this means that
\[(d-1)^2+4d-4 \geq 4t_{2} + 8t_{3} + 12t_{4} + 16t_{5} = 2t_{2}+2t_{3}-4t_{5}+d(d-1).\]
Hence
\begin{equation}
\label{n5}
\frac{3}{2}(d-1) + 2t_{5} \geq t_{2}+t_{3}.
\end{equation}
Moreover, equation \eqref{cruc} gives us yet another non-trivial bound on the intersection points, namely
\begin{equation}
\label{geq}
\frac{1}{4}(d^{2}+2d-3) \geq t_{2}+2t_{3}+3t_{4}+4t_{5}.
\end{equation}
Let us come back the na\"ive combinatorial count:
\begin{multline*}
\frac{d(d-1)}{2} = (t_{2} + 2t_{3} + 3t_{4}+4t_{5}) +t_{3} + 3(t_{4} + 2t_{5}) \stackrel{\eqref{geq}}{\leq} \frac{1}{4}(d^2+2d-3) + t_{3} + 3(t_{4}+2t_{5}) \\ \stackrel{\eqref{Melchior}}{\leq} \frac{1}{4}(d^2+2d-3)+t_{3} + 3(t_{2}-3) \leq \frac{1}{4}(d^{2}+2d-3)-9 + 3(t_{2}+t_{3}) \\
\stackrel{\eqref{n5}}{\leq}  \frac{1}{4}(d^{2}+2d-3)-9 + \frac{9}{2}(d-1)+6t_{5},
\end{multline*}
and hence we get 
$$t_{5} \geq \frac{1}{24}(d-3)(d-19).$$
On the other hand we have an upper-bound of the form
$$t_{5} \leq \frac{d(d-1)-16}{25},$$
hence we get
\[\frac{1}{24}(d-3)(d-19) \leq t_{5} \leq \frac{1}{25}(d(d-1)-16),\]
which gives us $d \in [8,9,10, \ldots , 522],$ and this completes the proof.
\end{proof}
\section{The boundedness of real plus-one generated line arrangements with points of multiplicity at most four}
\begin{theorem}\label{c}
Let $\mathcal{L} \subset \mathbb{P}^{2}_{\mathbb{R}}$ be an arrangement of $d\geq 3$ lines such that $t_{d}=0$ and it admits only double, triple and quadruple intersection points. Assume that $\mathcal{L}$ is plus-one generated. Then $d \leq 47$.
\end{theorem}
\begin{proof}
Our proof goes along similar lines as for Theorem \ref{main}. First of all, the assumption that $\mathcal{L}$ is plus-one generated implies that 
$$t_{2} + 4t_{3}+9t_{4} = d(d-1) - t_{2} - 2t_{3} - 3t_{4} = (d-1)^2 - d_{1}(d-d_{1}-1)-h,$$
where $h:=d_{3}-d_{2}+1\geq 1$.
This leads us to
$$d_{1}^{2}-d_{1}(d-1)+(t_{2}+2t_{3}+3t_{4}-d+1-h)=0.$$
The condition that $\mathcal{L}$ is plus-one generated implies that the discriminant $\triangle_{d_{1}} \geq 0$, and hence we get
$$d^{2}+2d-3+4h \geq 4t_{2}+8t_{3} + 12t_{4} = 2(t_{2}+t_{3})+d(d-1).$$
Summing up, we reach the following inequality:
\begin{equation}
\label{eee}
t_{2}+t_{3}\leq \frac{3(d-1)}{2}+2h.
\end{equation}
Let us come back to the na\"ive combinatorial count, we have
\begin{multline*}
d(d-1)=2t_{2}+6t_{3}+12t_{4}  \stackrel{\eqref{Melchior}}{\leq} 14t_{2}+6t_{3}-36 \leq 14(t_{2}+t_{3})-36 \stackrel{\eqref{eee}}{\leq}21(d-1)+28h-36,
\end{multline*}
and hence we get $$d^{2}-22d+57-28h \leq 0.$$
This gives us that $d\leq 11 + 2\sqrt{16+7h}.$
Since for line arrangements we have $h \leq d-2$ (\cite{HSch}, Corollary 3.5) we finally obtain the following inequality
\[d \leq 11 + 2\sqrt{2+7d},\]
which gives us $d\leq 47$, and this completes the proof.
\end{proof}

\section*{Acknowledgement}
The author would like to thank Piotr Pokora for his guidance and valuable remarks related to this work.

 Marek Janasz is supported by the National 
 Science Centre (Poland) Sonata Bis Grant  
 \textbf{2023/50/E/ST1/00025}. For the 
 purpose of Open Access, the author has 
 applied a CC-BY public copyright license 
 to any Author Accepted Manuscript (AAM) 
 version arising from this submission.

\bigskip
Affiliation of the author:
\noindent
Department of Mathematics,
University of the National Education Commission Krakow,
Podchor\c a\.zych 2,
PL-30-084 Krak\'ow, Poland. \\
\nopagebreak
\noindent
Marek Janasz: \texttt{marek.janasz@uken.krakow.pl} 

\begin{thebibliography}{9}

\bibitem{Abe18}  T. Abe, Plus--one generated and next to free arrangements of hyperplanes. \textit{Int. Math. Res. Not.} \textbf{2021(12)}: 9233 -- 9261 (2021).

\bibitem{her}
T. Abe, A. Dimca, and P. Pokora, A new hierarchy for complex plane curves. \textit{Canad. Math. Bull.} Advance Publication 1 -- 24 (2025), \url{https://doi.org/10.4153/S0008439525101422}.

\bibitem{RCS}  
A. Dimca,  \textit{Topics on real and complex singularities. An introduction.} Advanced Lectures in Mathematics. Braunschweig/Wiesbaden: Friedr. Vieweg \& Sohn. 1987.

\bibitem{Dimca}
A. Dimca,  \textit{Hyperplane arrangements. An introduction}. Universitext. Cham: Springer. xii, 200 p. (2017).

\bibitem{DS20} 
A. Dimca and G. Sticlaru, Plane curves with three syzygies, minimal Tjurina curves, and nearly cuspidal curves. \textit{Geom. Dedicata}  {\bf 207}:  29 -- 49 (2020).

\bibitem{DimPok} 
A. Dimca and P. Pokora, Maximizing Curves Viewed as Free Curves. \textit{Int. Math. Res. Not. IMRN} \textbf{22}:  19156 -- 19183 (2023).

\bibitem{duP}
A. Du Plessis and C. T. C. Wall, Application of the theory of the discriminant to highly singular plane curves. \textit{Math. Proc. Camb. Philos. Soc.} \textbf{126(2)}: 259 -- 266 (1999).

\bibitem{Melchior}  E. Melchior, \"{U}ber Vielseite der Projektive Ebene. Deutsche Mathematik {\bf 5}: 461 – 475 (1941).

\bibitem{JanLes05}
M. Janasz and I. Leśniak, On free line arrangements with double, triple and quadruple points. \textit{European Journal of Mathematics} \textbf{11}: Art. Id. 81 (2025).

\bibitem{Saito}
K. Saito, Theory of logarithmic differential forms and logarithmic vector fields. \textit{J. Fac. Sci., Univ. Tokyo, Sect. I A} \textbf{27}: 265 -- 291 (1980).

\bibitem{HSch} H. Schenck, Elementary modifications and line configurations in $\PP^2$. Comment. Math. Helv. 78: 447 -- 462 (2003). 

\bibitem{Sh}
I. N. Shnurnikov, A $t_{k}$ inequality for arrangements of pseudolines. \textit{Discrete Comput. Geom.} \textbf{55(2)}: 284 -- 295 (2016).
\end{thebibliography}
\end{document}